\documentclass[11pt,onecolumn]{article}
\usepackage{amscd}
\usepackage{times}
\usepackage{amsmath}
\usepackage{amssymb}
\usepackage{xspace}
\usepackage{theorem}
\usepackage{graphicx}
\usepackage{ifpdf}
\usepackage{url,hyperref}
\usepackage{latexsym}
\usepackage{euscript}
\usepackage{xspace}
\usepackage[all]{xy}
\usepackage{color}
\usepackage{makeidx}
\usepackage{tikz}
\usepackage{enumitem}
\long\def\remove#1{}
\newtheorem{theorem}{Theorem}[section] % section
\newtheorem{lemma}[theorem]{Lemma}

\newtheorem{obs}[theorem]{Observation}

\newtheorem{proposition}[theorem]{Proposition}
\newenvironment{proof}{{\em Proof:}}{\hfill{\hfill\rule{2mm}{2mm}}}

\newcommand {\mm}[1] {\ifmmode{#1}\else{\mbox{\(#1\)}}\fi}

\newcommand{\cancel}[1]

\begin{document}

\title {Note on the conjugacy classes of elements and their centralizers  for the free product of two groups }

\author{
Dan Burghelea  \thanks{
Department of Mathematics,
The Ohio State University, Columbus, OH 43210,USA.
Email: {\tt burghele@math.ohio-state.edu}}
}
\date{}

\date{}
\maketitle

\begin{abstract}
We describe the conjugacy classes of the elements of the free product  of two groups and their centralizers  and, as a consequence, we correct the calculation of the cyclic and periodic cyclic homology  of the group ring of the free product of two groups given in a previous paper. 
\end{abstract}

\setcounter{tocdepth}{1}
%\tableofcontents
\section {Introduction}

This Note was prompted  by a mistake pointed out  by Markus Land about the cyclic and periodic cyclic homology of the group ring of the free product, precisely 
 Propositions II and II{p}  in the paper {\bf The cyclic homology of the group rings } published in Comment. Math. Helv., 60 (1985)  no 3, 354-365 .
 The mistake was the result  of a miscalculation of the centralizers of the conjugacy classes of elements of the free product of two groups. 
 In this note we provide a correct description of them and, as a consequence, correct the statements of Proposition II and IIp in \cite{B}.   
 
In consistency with the notation in \cite {B} for a group $G$ and an element $x\in G$ one denotes by $G_x:= \{y\in G \mid  y\cdot x= x\cdot y\}$ the centralizer of $x,$ 
 by $\{x\}$ the subgroup 
 generated by the element $x$ and by $N_x$ the quotient group $N_x:= G_x / \{x\}.$  These groups remain isomorphic for all $x$ in the same conjugacy class. 
 Denote by  $\langle G \rangle$ the set of conjugacy classes of elements of  $G$ and for  $x\in G$ write $\hat x$ for the conjugacy class of $x$.
For the  groups $H$ and $G$ one denotes the nonzero elements  by $h$ and $g$  and the  
neuter elements by $e_H$ and $e_G.$ 

Consider the free product $P= H\ast G.$ 
Any element of $x\in P$ is representable  (not uniquely) by a {\it word}  $s_1 s_2 \cdots s_r$ with $s_i\in H\sqcup G.$  The product of the elements $x$ and $x',$ represented by the words $s_1 s_2\cdots s_r$ and  $s'_1 s'_2 \cdots s'_{r'}$ 
is representable by the concatenation $s_1 s_2 \cdots s_r s'_1 s'_2 \cdots s'_{r'}.$ 
Modifying a word representation  of an element $x\in P$ by

i)  removing all elements $s_i$ of the form $e_H$ and $e_G,$ 

ii)  replacing consecutive elements $\cdots s_i s_{i+1}\cdots$ by their product when in either $H$ or $G,$ 

\noindent leads to  a smaller word representations of $x\in P,$ the {\it reduced word }representation, which is unique.  (The empty word is the reduced representation of $e_P$). 

The reduced word representation   $s_1 s_2\cdots s_r$  for the element   $x \in P$ is characterized by 

a)   $s_i\in (H\setminus e_H) \sqcup (G\setminus e_G)$

b)    consecutive $ s_i $ and $s_{i+1}$ belong to different groups.

Consequently,  a  nontrivial  elements of $x\in P$ has  an unique {\it reduced word} representation  of one of the  following  seven types: 

Type 1:  $w=\ \  h_1 g_1 h_2  g_2 \cdots h_k g_k,$ \ \ \ \  $k\geq 1,$ 

Type 2:  $w=\ \  h_1 g_1  h_2 g_2 \cdots h_k g_k \ h,$ \ \  $k\geq 1,$ 

Type 3:  $w= g \ h_1 g_1  h_2  g_2 \cdots h_k  g_k,$ \ \ \ \ $k\geq 1,$ 

Type 4:  $w= g \ h_1 g_1 h_2  g_2 \cdots h_k  g_k\  h,$ \ \  $k\geq 1,$ 

Type 5:    $w= g,$

Type 6:    $w= h,$

Type 7:    $w= g h.$
\vskip .1in

Because of the unicity of the reduced word representation the following facts hold true:
\begin{enumerate} 
\item  if two nontrivial elements commute then they belong  to the same type;
as a consequence  if $x$ is represented by a reduced word of type 5 or type 6 then $P_x= G_x$ or $P_x= H_x.$

\item  any nontrivial element $x\in P$ of type different from type 5 and type 6  is conjugate to an element of type 1. 
\item the following proposition holds true
\begin {proposition}\label {P1}
Suppose   $w'= h_{1}\  g_1\  h_2 \  g_2 \cdots h_r \  g_r$ and $w'= h'_1\  g'_1\  h'_2 \  g'_2 \cdots h'_r\   g'_{r'}$ are two reduced words of type 1, representing elements $x$ and $x'$ in $P$ s.t. $x\cdot x'= x'\cdot x$  and suppose $c$ is the greatest common divisor of $r$ and $r'.$ 
Then there exists a reduced word of type 1, $w_0= h_{1}''  \  g_1''\   h_2 ''\  g_2'' \cdots h_c''\   g_c'',$ such that $w$ is the concatenation of $r/c$ copies of $w_0$ and $w'$ is the concatenation of $r'/c$ copies of $w_0.$
\end{proposition}
\end{enumerate}

\begin{obs}\label {O2}\

Items 1 and 2  above imply that $\langle P\rangle =e_P \sqcup (\langle H \rangle \setminus e_H)\sqcup (\langle G\rangle \setminus e_G)\sqcup U$ with
$U= \{ \hat x\in \langle P \rangle \mid \hat x\cap (e_H \ast G) = \emptyset, \hat x \cap (H\ast e_G)=  \emptyset\}$ and the centralizers in $P$ of the elements in $e_H\ast G\subset P$ resp. in $H\ast e_G\subset P$ remain the same as the centralizers in $G$ and $H$ resp..  

Item 3 ( i.e. Proposition (\ref{P1})  shows that for $x\in \hat x\in U, $  the pair {\it group-subgroup}  $(P_x, \{x\})$ is isomorphic to the pair $(\mathbb Z, k(x)\mathbb Z),$ 
hence  $N_x \simeq \mathbb Z_{k(x)}.$ 
Here $k(x)$ is the largest integer $k$ s.t. $x= y^k.$ 
\end{obs}
\section {Proof of Proposition \ref{P1}}

Let $\mathbb S$ be a set of symbols.  Let 
$\mathcal S := \{s_1, s_2, \cdots, s_r\cdots,s_n\}$  be an  ordered set of symbols with $s_i\in \mathbb S$ (i.e. an word with letters in $\mathbb S$), 
 $p<n$ and $d=n-p.$ 

\begin{lemma} \label {L1}
Suppose that the collection  $\mathcal S$ satisfies: 
\begin{enumerate}  [label =(\roman*)]
\item $s_i= s_{d+i}$\  for $i\leq p,$
\item $s_i= s_{i+p}$ \ for $i\leq d.$ 
\end{enumerate}

1. If   $n$ and $p$ are relatively prime  
then all $s_i$ are equal.

2. If $c$ is the greatest common divisor  of $n$ and $p$ and $\mathcal S'= s_1, s_2, \cdots s_c$ is  the ordered set of the first $c$ symbols of $\mathcal S$ 
  then $\mathcal S$ is the concatenation of $n/c$ copies of $\mathcal S'.$ 
\end{lemma}

\begin{proof}
For any  
 $r=1,2,\cdots,  d$ \  define the subsets of $\mathcal S(r)\subset \mathcal S,$
consisting of all elements of $\mathcal S$ indexed by  $r+k d $ for $k=0,1, \cdots,$ namely 
$$\mathcal S(r):= \{ s_r, s_{r+d}, s_{r+2d}, \cdots  s_{r+kd}\cdots \}. $$ 
Note that the sets $\mathcal S(r)$ are disjoint and their union is $\mathcal S.$

\medskip 

Proof of item 1:
For each $\mathcal S(r)$ let $k_r$ be the unique integer such that 
$r+(k_r-1)d\leq p < r+k_r d.$  The first inequality guaranties that 
\begin{equation}\label{e1}
r+ k_r d-p \leq d.
\end{equation}
In view of the hypothesis (i)
all  elements of  $\mathcal S(r)$  are equal and in view of the hypothesis (ii)
the elements of the collections 
$S(r)$ and  $\mathcal  S(r+ k_r d-p)$  are equal; 

Consider the pairs of integers  $(r_i, \kappa _i)$ with $1\leq i \leq d$ defined inductively by: 

a) $r_1=1$ and $\kappa_1:= k_1$ 

b) $r_{i+1}= 1+(\kappa_1 + \kappa_2 +\cdots \kappa_i) d-  i p$  and $\kappa_{i+1}= k_{ r_{i+1}}$ 

\noindent  Note that in view of inequality (\ref{e1}) all $r_i \leq d.$ 

Let $\mathcal S_i= :\mathcal S(r_i).$ Since $p$ and $d$ are relatively prime $\mathcal S_i$ and $\mathcal S_j$ for $i\ne j$ can never be the same since $d$ can not divide $i-j.$  Then  the sets $\mathcal S_i, i=1,2, \cdots d$ 
provide a permutation of the sets $\mathcal S(r), r=1,2, \cdots d.$  Hypothesis (ii) implies that the elements of $\mathcal S_i$ and $\mathcal S_{i+1}$ are equal for any $i$ hence all elements of $\mathcal S$ are equal.
\medskip 

Proof of item 2:  
Consider the set of symbols $\mathbb S'= \mathbb S\times \mathbb S\times \cdots \mathbb S,$  the $c-$ fold  cartesian product of $\mathbb S;$ clearly $\mathcal S' \in \mathbb S'.$  Interpret $\mathcal S$ as an ordered set of $n/c$ symbols of  $\mathbb S'.$  Clearly item 1. implies item 2.

\end{proof}
 
To prove Proposition \ref{P1}  consider the set  of symbols $\mathbb S= (H\setminus e_H) \times (G\setminus e_G)$ and write $w= s_1 \ s_2 \ \cdots \ s_n,$ and $w'= s'_1 \ s'_2 \ \cdots \ s'_{p},$ 
Since the concatenation $w w'$ and $w' w$ are the same one has
 \begin {enumerate} 
 \item $s'_i= s_i$ for $i\leq p$
\item $s'_i= s_{d+i}$ for $i\leq p$
\item $s_i= s_{i+p}$ for $i\leq n-p=d$
\end{enumerate}
which implies 

$s_i= s_{d+i}$  for all $i\leq p$ and

$s_i = s_{p+i}$ for $i\leq d,$

\noindent which in view  of Lemma (\ref{L1}) implies that $w$ is the concatenation exactly $n/c$ times of $w_0=s_1 s_2\cdots, s_c$ and in view of  the equality $s_i= s'_i$ for $i\leq d,$ $w'$ is the concatenation of exactly $p/c$ copies of $w_0.$ 

\section {Cyclic resp. periodic cyclic homology of the group-ring $R[P],$}

Let $R$ be a commutative ring with unit and $G$ a group. 
Recall that the reduced cyclic  resp. periodic cyclic homology, $\tilde HC_\ast(R[G])$ resp. $\tilde {PHC}_\ast(R[G]),$ of the group ring $R[G]$ 
is the co-kernel  of the split injective map $HC_\ast (R[e_g])= HC_\ast (R) \to HC_\ast(R[G])$ resp. 
$PHC_\ast (R[e_G])= PHC_\ast (R) \to PHC_\ast(R[G])$  induced by the inclusion of the trivial subgroup $e_G$ to $G.$ One refers to the cyclic  resp. periodic cyclic homology of the group ring $R[G]$ as the {\it unreduced} version of these homologies. Clearly,  the unreduced version  is the direct sum of the reduced version with one copy of the cyclic resp. periodic cyclic homology of $R.$

As shown in \cite{B} 
all these homologies, reduced or unreduced, say $\mathcal H_\ast(R[G]),$ are graded $R-$modules which are  direct sums of graded $R-$modules $\mathcal H_\ast(R[G])_{\hat x}$ indexed by the conjugacy classes $\hat x\in \langle G \rangle,$
referred to as the contribution of $\hat x,$ $$\mathcal H_\ast(R[G])= \oplus_{\hat x\in \langle G\rangle} \mathcal H(R[G])_{\hat x}.$$  

For each $\hat x\ne e_G$ the contribution to the reduced and unreduced version are the same but for $e_G$  the  unreduced version is equal to the reduced version direct sum   $\mathcal H_\ast (R).$  

For each conjugacy class $\hat x$  one defines $n(\hat x):= n(x)$ the order of the element $x$  and $k (\hat x):= \kappa(x)$ the largest $k$ s.t. $x= y^k$; clearly $n(x)$ and $\kappa(x)$ are the same for all $x$ in the same conjugacy class. 

Recall from \cite {B} the notations:
\begin{enumerate}
\item 
$$ K_\ast (R[G]):= \begin{cases}  \oplus_{n \geq 0}H_{2n} (BG; R) \ \rm {if} \  \ast=\rm {even}\\ \oplus _{n\geq 0} H_{2n+1}(BG; R)\ \rm {if} \  \ast= \rm {odd} \end{cases}\ \   \tilde K_\ast (R[G]):= \begin{cases}  \oplus_{n > 0}H_{2n} (BG; R) \ \rm {if} \  \ast=\rm {even}  \\ \oplus _{n\geq 0} H_{2n+1}(BG; R)\ \rm {if} \  \ast=\rm {odd} \end{cases}                   $$

\item 
for $x \in G$  with  $n(x)=\infty$  
$$ T_\ast (\hat x; R)= T_\ast (x; R):= \lim ( \xymatrix { \cdots \ar[r] & H_{\ast +2n} (BN_x: r)\ar[r] ^S & H_{\ast+ 2n -2} (BN_x; R) \ar[r] & \cdots}$$
\end{enumerate}
 with $S$ the Gysin homomorphism of the fibration $B\{x\}= S^1 \to BG_x \to  BN_x$ where $N_x= G_x/ \{x\},$ which up to isomorphism depends only on the conjugacy class of $x$ and then denoted by $T(\hat x; R).$

Recall from \cite {B} that the 
 contribution of  $\hat x$ when  $0\ne n(\hat x) <\infty$ is 
$$HC_\ast (R[G])_{\hat x}= H_\ast ( B(N_{\hat x})\times BS^1\times K(\mathbb Z_{n(\hat x)}, 1); R), \quad  PHC_\ast (R[G])_{\hat x}= K_\ast(R[ N_{\hat x}])  $$  and  when  $n(\hat x) =\infty$ is 
$$HC_\ast (R[G])_{\hat x}= H_\ast ( B(N_x) ; R), \quad  PHC_\ast (R[G])_{\hat x}= T_{\ast} (\hat x; R)$$  while for  $n(\hat x)=0,$ hence $\hat x= e_G,$  is   
$$HC_\ast (R[G])_{e_G}= H_\ast ( B(G)\times BS^1; R),  \quad    PHC_\ast (R[G])_{e_G}= K_\ast (R[G])$$  and  $$\tilde  HC_\ast (R[G])_{e_G}= H_\ast ( B(G)\times BS^1/ \ast\times BS^1; R).\quad   \tilde {PHC}_\ast (R[G])_{e_G}= \tilde K_\ast (R[G])$$ 
In particular one has 

\begin {proposition} \label {P3} \  

$$\tilde {PHC}_\ast (R[G])= \tilde K_\ast (R[G])  \bigoplus ( \oplus _{\hat x\in (\langle G \rangle' \setminus {e_G})}K_\ast(BN_{\hat x}; R)) \bigoplus (\oplus_{\hat x\in \langle G\rangle''} T_\ast (\hat x; R))$$

where $\langle G\rangle ':= \{\hat x\in \langle G\rangle \mid n(\hat x)<\infty\}$ and $\langle G\rangle ''= \{\hat x\in \langle G\rangle \mid n(\hat x)= \infty\}.$
\end{proposition}\label {P3}
An equivalent form of this proposition  is stated  in \cite {B} for a field of characteristic zero as Theorem 1'. 
\medskip
\vskip .2in
Let $H$ and $G$ be two groups and $P= H\ast G$ their free product. Recall that $B(H\ast G)= BH \vee BG$  the base point union of the spaces $BH$ and $BG.$   

As an immediate consequence of  Observation (\ref{O2}) (description of conjugacy classes of elements of $P$ and of  their centralizers) in the previous section one has 

\begin {proposition}   \label {P4}\

 \noindent $$\tilde HC_\ast(R[H\ast G]) = \tilde HC_\ast(R[H]) \bigoplus \tilde HC_\ast (R[G])\bigoplus (\oplus_{\hat x\in U}H_\ast ( B\mathbb Z_{k(\hat x)}; R))  $$
$$\tilde {PHC}_\ast(R[H\ast G]) = \tilde {PHC}_\ast(R[H]) \bigoplus \tilde {PHC}_\ast (R[G])\  \bigoplus (\oplus_{\hat x\in U} T_\ast ( \hat x;  R))  $$

where 
 $$H_\ast(B(\mathbb Z_k);R)= 
\begin{cases} 
%\begin{aligned} 
R \ \rm {for} \ \ast=0\\
H_1( B(\mathbb Z_k); R) \ \rm {for} \  \ast \rm {odd} \\   
H_2(B (\mathbb Z_k); R) \ \rm{for} \ast \  \rm{even}\ne 0 
\end{cases}$$%\end{aligned} 

 and 
 $$ T_\ast (\hat x; R)= \begin{cases} H_1(B(\mathbb Z_{k(x)}); R) \ \rm {for} \  \ast \rm {odd} \\   
H_2(B (\mathbb Z_{k(x)}); R) \ \rm{for} \ast \  \rm{even} \end{cases}$$
 \end{proposition} 
Note that if $R$ is an algebra over a field of characteristic zero then  $H_\ast(B(\mathbb Z_k); R)$ is concentrated in degree zero and isomorphic to $R$ and  for any $\hat x\in U,$ $T_\ast (\hat x; R)$ vanishes. 

To correct all inaccuracies in \cite {B} we insert the following errata to \cite{B}

\section {Errata to the paper  The cyclic homology of the group rings, Comment. Math. Helv., 60,
                         1985, no. 3, 354-365 }
 In the paper {\bf The cyclic homology of the group rings } published in Comment. Math. Helv., 60 (1985)  no 3, 354-365
 Propositions II and II{p}, straightforward consequences of the main result,Theorem I, are not true as stated. The statements become  correct  provided the cyclic resp. periodic cyclic homology are replaced by their reduced versions  and the ring $R$  is a  $\mathbb Q-$algebra, for instance a field of characteristic zero. The reduced cyclic resp. periodic cyclic homology of $ R[G]$ is quotient of the obvious split injective maps,  $i_G: HC_\ast( R)\to HC_\ast (R[G])$ resp.  $i_G: PHC_\ast(R)\to PHC(R[G])$
induced by the inclusion $e_G \in G.$
 Also on   line 1 page 363 to make the statement correct one shall replace     
  "$P_x\ne \{x\}$" by "$N_x\ne \mathbb Z_{k(x)}$ with   $\mathbb Z_k$ denoting  the finite cyclic group of order $k,$ and $k(x)$    the largest integer $k$ s.t. $x= y^k.$ 
%  This  integer  depends only on the  conjugacy class of $x.$} $". 
\vskip .2in 

 I thank {\bf Markus Land }for bringing this to my attention and suggesting  the use of reduced cyclic and periodic cyclic homology for group rings. 

For an arbitrary commutative  ring  with unit $R,$  Propositions II and IIp should be corrected as follows
\begin{enumerate}
\item In Proposition II, $HC_\ast$ has to be replaced by the reduced version $\tilde HC_\ast$ and  the sentence 
  {\it " $\ R_{\hat \alpha} = R$ regarded as a graded module concentrated in the degree zero"}   
by  {\it "$R_{\hat \alpha}= H_\ast ( B(\mathbb Z_{k(x)}); R),  x\in \hat \alpha .$" }
\item In Proposition IIp,  $PHC_\ast$ has to be replaced by its  reduced version, $\tilde {PHC}_\ast,$ and to the right side of the equality one should add 
$\oplus _{\hat x\in U} T_\ast (\hat x; R),$ with 
$U= \{ \hat x\in \langle H\ast G \rangle \mid \hat x\cap (e_H\ast G )= \emptyset, \hat x \cap (H\ast e_G)= \emptyset\}$ where $\langle \Gamma \rangle$ denotes the set of conjugacy classes of elements of the group $\Gamma$ and $e_\Gamma$ the neuter element of $\Gamma.$
\end{enumerate}
Note that 
$$H_\ast(B(\mathbb Z_k);R)= 
\begin{cases} 
%\begin{aligned} 
R \ \rm {for} \ \ast=0\\
H_1( B(\mathbb Z_k); R) \ \rm {for} \  \ast \rm {odd} \\   
H_2(B (\mathbb Z_k); R) \ \rm{for} \ast \  \rm{even}\ne 0 
%\end{aligned} 
\end{cases}$$
and for   $\hat x\in U,$  and $x\in \hat x$  the pair of group - subgroup  $(G_x, \{x\})$ is isomorphic to the pair $ (\mathbb Z, k(x)\mathbb Z),$ hence $N_x= \mathbb Z_{k(x)},$ and 
$$ T_\ast (\hat x; R)= \begin{cases} H_1(B(\mathbb Z_{k(x)}); R) \ \rm {for} \  \ast \rm {odd} \\   
H_2(B (\mathbb Z_{k(x)}); R) \ \rm{for} \ast \  \rm{even} \end{cases}
\footnote {Recall  that  $T_\ast (\hat x; R)\simeq T_\ast ( x; R):= \varprojlim _n  (\xymatrix{\cdots \ar[r]& H_{\ast +2n} (B (\mathbb Z_ {k(x)}) ; R) \ar[r]^\Sigma & H_{\ast +2n-2}( B(\mathbb Z_ {k(x)}); R) \ar[r]  &\cdots })$ 
 { where $\Sigma$ is the isomorphism in the homology Gysin sequence of the fibration   $S^1= B(k \mathbb Z) \to B(Z) \to B(Z_k)$ } which vanishes when $R$ is a $\mathbb Q-$ algebra.}.$$

\end{document}